\documentclass[a4paper,11pt]{amsart} 
\usepackage{amssymb, amsmath}
\usepackage{amscd}
\numberwithin{equation}{section}
\usepackage{epsfig}
\usepackage{amsmath}
\usepackage{amsfonts,amssymb,amsopn}

\usepackage{amsthm}
\usepackage{hyperref}
\usepackage{verbatim}
\usepackage{color}
\usepackage[color,all]{xy}

\newcommand{\cA}{A}

\newcommand{\cE}{{\mathcal E}}
\newcommand{\cF}{{\mathcal F}}

\newcommand{\cI}{{\mathcal I}}
\newcommand{\cL}{L}
\newcommand{\cO}{{\mathcal O}}
\newcommand{\cP}{{\mathcal P}}
\newcommand{\cQ}{{\mathcal Q}}
\newcommand{\cU}{{\mathcal U}}
\newcommand{\cV}{{\mathcal V}}

\newcommand{\C}{{\mathbb C}}
\newcommand{\PP}{{\mathbb P}}
\newcommand{\bT}{{\mathbb T}}
\newcommand{\bZ}{{\mathbb Z}}

\newcommand{\tW}{\widetilde{W}}

\newcommand{\Ker}{\mathrm{Ker}}
\newcommand{\Image}{\mathrm{Im}\,}

\newcommand{\Hom}{\mathrm{Hom}}
\newcommand{\rank}{\mathrm{rk}\,}
\newcommand{\Pic}{\mathrm{Pic}}
\newcommand{\ev}{\mathrm{ev}}
\newcommand{\isom}{\xrightarrow{\sim}}

\newcommand{\Oy}{{{\cO}_{Y}}}
\newcommand{\Kc}{K_{C}}
\newcommand{\Oc}{{{\cO}_{C}}}
\newcommand{\Ox}{{{\cO}_{X}}}

\newcommand{\gen}{{\mathrm{gen}}}

\newcommand{\Spec}{\mathrm{Spec}\,}

\newcommand{\Gr}{\mathrm{Gr}}

\newcommand{\st}{\mathrm{st}}

\newcommand{\pe}{\PP E^\vee}
\newcommand{\ope}{\cO_{\pe}}
\newcommand{\opeo}{{\ope (1)}}
\newcommand{\pf}{\PP F^\vee}

\newcommand{\opfo}{\cO_{\pf}(1)}
\newcommand{\pv}{\PP V^\vee}
\newcommand{\mlv}{M_{\cL, V}}

\newcommand{\moev}{M_{\opeo, V}}
\newcommand{\mofw}{M_{\opfo, W}}
\newcommand{\grdn}{G_0 (r, d, n)}

\newtheorem{theorem}{{\textbf Theorem}}[section]
\newtheorem{proposition}[theorem]{{\textbf Proposition}}
\newtheorem{corollary}[theorem]{{\textbf Corollary}}
\newtheorem{lemma}[theorem]{{\textbf Lemma}}
\newtheorem{remit}[theorem]{{\textbf Remark}}

\newcounter{tmp}

\title{Stability of kernel bundles on projective bundles over curves}

\author{Abel Castorena}
\address{Centro de Ciencias Matem\'aticas -- UNAM Campus Morelia, Antigua Carretera a P\'atzcuaro \# 8701, Col.\ Ex Hacienda San Jos\'e de la Huerta, Morelia, Michoac\'an, Mexico C.\ P.\ 58089.}
\email{abel@matmor.unam.mx}

\author{George H.\ Hitching}
\address{Oslo Metropolitan University, Postboks 4, St. Olavs plass, 0130 Oslo, Norway.}
\email{gehahi@oslomet.no}



\begin{document}

\begin{abstract}
Let $X$ be a projective bundle over a smooth curve $C$ of genus $g \ge 3$, and consider the relative hyperplane bundle $\Ox (1) \to X$. Let $V \subseteq H^0 ( X , \Ox (1) )$ be a generating subspace. We prove that when $C$, $X$ and $V$ are general in moduli and $\Ox (1)$ is sufficiently ample, the kernel bundle of the system $(\Ox (1) , V)$ is $\Ox (1)$-stable.
\end{abstract}

\maketitle

\section{Introduction}

Let $X$ be a projective variety and $\cL \to X$ a generated line bundle. For each generating subspace $V \subseteq H^0 (X, \cL )$ one has the exact sequence
\[
0 \ \to \ \mlv \ \to \ \Ox \otimes V \ \to \ \cL \ \to \ 0 .
\]
The bundle $\mlv$ is called the \textsl{kernel bundle} or \textsl{syzygy bundle} of the pair $( \cL , V)$. Fixing an ample bundle $\cA \to X$, it is a fundamental question to determine whether $\mlv$ is $\cA$-stable. (Typically, $A$ is a rational multiple of $\cL$.)

When $X$ is a curve, this is the well known \textsl{Butler conjecture} and dates back at least to \cite{But94, But97} (where in fact $\cL$ is replaced with an arbitrary rank bundle over $X$). This has been settled for $X$ and $( \cL , V)$ general in moduli in \cite{BN} (see also \cite{BBN}). Stability of $\mlv$ in various cases has found application to Brill--Noether loci in \cite{BN} and many other works, and to generalised theta divisors in \cite{Popa} and \cite{Bea}.

If $V = H^0 (X, \cL )$, we abbreviate $\mlv$ to $M_\cL$. For $\dim X \ge 2$, the $\cA$-stability of $M_\cL$ has been established in various cases for surfaces \cite{ELM} and \cite{TZ}; varieties of Picard number one \cite{ELM} and \cite{Sar}; abelian varieties \cite{CL}, and in general for $h^0 (X, \cL )$ sufficiently large \cite{Rek}. When $V$ may be a proper subspace of $H^0 (X, \cL )$, somewhat less is known; but in \cite{MS}, for $X = \PP^m \times \PP^n$, the $\cL$-stability of $\mlv$ is proven for many choices of $( \cL , V)$, and this is applied to questions of rigidity of $\mlv$ in moduli.

In the present work, we prove the stability of $\mlv$ in another situation where $V \subseteq H^0 (X, \cL )$ may be a proper subspace. 

\begingroup
\setcounter{tmp}{\value{theorem}}
\setcounter{theorem}{0}
\renewcommand\thetheorem{\Alph{theorem}}

\begin{theorem} \label{main}
Let $C$ be a curve of genus $g \ge 3$ which is general in moduli. Fix integers $r$, $n$ and $d$ satisfying
\begin{equation} \label{NumHyp}
n \ \ge \ r + 1 \quad \hbox{and} \quad d \ \ge \ (n + r - 1)(g - 1) + (n - 1) .
\end{equation}
Then for a general bundle $E \to C$ of rank $r$ and degree $d$ and a general generating subspace $V \subseteq H^0 ( \pe , \opeo )$, the kernel bundle $\mlv \to \pe$ is $\opeo$-stable.
\end{theorem}

\endgroup

The meaning of ``general'' will be made more precise in {\S} \ref{section:Irred}, where the irreducibility of certain moduli spaces of coherent systems over $C$ is established. Theorem \ref{main} is proven by induction on $r$. Butler's conjecture as proven in \cite{BN} furnishes the case $r = 1$. To reduce from dimension $r$ to $r - 1$, we utilise the fact that a projective bundle $\pe \to C$ admits divisors which are themselves projective bundles with similar properties to $\pe$. A strategy inspired by \cite{Cam} and \cite{ELM}, together with results of \cite{Mis}, allows us to analyse the restriction of $\mlv$ to such a projective subbundle and infer the desired stability statement.


\subsection*{Acknowledgements}

The first author is supported with grant PAPIIT IN101226 ``Teor\'{\i}a de Brill--Noether y aplicaciones'' from UNAM, M\'exico. He also thanks Oslo Metropolitan University for hospitality during a visit in October 2025.

\subsection*{Notation}

We work throughout over the field $\C$. If $A \times B$ is a product, we denote by $\pi_A$ and $\pi_B$ respectively the projections to the factors $A$ and $B$. If $E$ is a vector bundle, we write $\pe$ for the projective bundle of hyperplanes in fibres of $E$.

\section{Preliminaries} \label{section:Prelims}

In this section we recall familiar material on stability of torsion free sheaves and on bundles and coherent systems over curves.

\subsection{Stability of torsion free sheaves} \label{subsection:Stability}

Let $X$ be a smooth projective variety of dimension $r \ge 1$, and $p \colon X \to \Spec \C$ the structure morphism. For a $0$-cycle $\gamma \in A_0 (X)$, as in \cite[p.\ 13]{Ful} we write
\[
\int_X \gamma \ := \ p_* \gamma \ \in \ A_0 ( \Spec \C ) \ = \ \bZ .
\]
For $\cA \to X$ an ample line bundle, $c_1 ( \cA ) \in A^1 (X)$ is the polarisation associated to $\cA$. Then the \textsl{$\cA$-degree} and \textsl{$\cA$-slope} of a torsion free sheaf $F$ on $X$ are defined as
\begin{equation} \label{DefnDegSlope}
\deg_\cA F \ := \ \int_X c_1 (F) \cdot c_1 (\cA)^{r-1} \quad \hbox{and} \quad \mu_\cA (F) \ := \ \frac{\deg_\cA F}{\rank F}
\end{equation}
respectively. Such an $F$ is \textsl{$\cA$-semistable} if $\mu_\cA (G) \le \mu_\cA (F)$ for all proper nonzero subsheaves $G$ of $F$, and \textsl{$\cA$-stable} if inequality is strict for all such $G$.

Note that if $\dim X = 1$, then $c_1 (F)$ is already a $0$-cycle. Here $\deg_\cA F$ reduces to $\int_X c_1 (F)$, which is independent of $\cA$. Thus we  write simply $\deg F$ and $\mu (F)$ in (\ref{DefnDegSlope}).

\subsection{Generated coherent systems}

By a \textsl{coherent system} over a variety $X$, we will understand a pair $(E, V)$ where $E \to X$ is a vector bundle and $V$ a subspace of $H^0 (X, E)$. Such an $(E, V)$ is \textsl{generated} if the evaluation map $V \to E|_x$ is surjective for all $x \in X$.

A \textsl{family of coherent systems over $X$ parametrised by a base $B$} is a pair $(\cE, \cV)$ where $\cE \to B \times X$ is a vector bundle and $\cV$ a locally free subsheaf of $\left( \pi_B \right)_* \cE$. The following two results are well known, but as we did not find a suitable reference, we give proofs.

\begin{lemma} \label{BgenOpen}
Let $X$ be a smooth complete variety. Let $(\cE, \cV )$ be a family of coherent systems over $X$ parametrised by a base $B$. Set
\[
B_\gen \ := \ \left\{ b \in B : \left( \cE_b , \cV_b \right) \hbox{ is generated} \right\} .
\]
Then $B_\gen$ is an open subset of $B$.
\end{lemma}

\begin{proof}
Recall the canonical evaluation map $\ev \colon \pi_B^* \cV \to \pi_B^* \left( \pi_B \right)_* \cE \to \cE$. As $\pi_B^* \cV$ is locally free on $B$, the set
\[
\left\{ (b, x) \in B \times X : \ev|_{(b, x)} \colon \cV_b \to \cE_b|_x \hbox{ is not surjective} \right\}
\]
is a closed (indeed, determinantal) subvariety of $B \times X$. Its image in $B$ is exactly the complement of $B_\gen$. As $X$ is complete, this image is closed. Thus $B_\gen$ is open in $B$.
\end{proof}

The following is an easy adaptation of \cite[Lemma 2.1 (i)]{ChH}.

\begin{lemma} \label{NonVan}
Let $X$ be a variety of dimension $m$. Let $E \to X$ be a vector bundle of rank $r \ge m + 1$ and $V \subseteq H^0 (X, E)$ a generating subspace. Then a general element of $V$ does not vanish at any point of $X$.
\end{lemma}

\begin{proof}
As $E$ is generated, for any $x \in X$ the subspace $V \cap H^0 (X, E \otimes \cI_x )$ of sections vanishing at $x$ has dimension $\dim V - r$. Thus the locus
\[
\bigcup_{x \in X} \left( V \cap H^0 (X, E \otimes \cI_x ) \right)
\]
of sections vanishing at some point of $X$ has dimension at most $\dim V - r + m < \dim V$. The statement follows.
\end{proof}


\subsection{Moduli of coherent systems over a curve} \label{subsection:CohSys}

Let $C$ be a projective smooth curve of genus $g \ge 3$. The \textsl{type} of a coherent system $(E, V)$ over $C$ is the triple $(\rank E , \deg E , \dim V )$. Let $\alpha$ be a positive real number. Then $(E, V)$ is said to be \textsl{$\alpha$-semistable} if for each proper subbundle $F \subset E$ and subspace $W \subseteq V \cap H^0 (C, F)$ we have
\[
\frac{\deg F + \alpha \cdot \dim W}{\rank F} \ \le \ \frac{\deg E + \alpha \cdot \dim V}{\rank E} .
\]
The system is \textsl{$\alpha$-stable} if inequality is strict for all $(F, W)$. By \cite{KN}, there is a quasiprojective moduli space $G (r, d, n; \alpha )$ parametrising $\alpha$-stable coherent systems of type $(r, d, n)$. 

\begin{lemma} \label{alphasstability}
Suppose that $\alpha$ is close to $0$. Let $(E, V)$ be a coherent system over $C$. If $E$ is a stable vector bundle, then $(E, V)$ is $\alpha$-stable. Conversely, if $(E, V)$ is $\alpha$-stable, then $E$ is a semistable vector bundle.
\end{lemma}

\begin{proof}
This is stated in \cite[p.\ 3]{But97}. Details can be found in \cite[Lemma 3.1]{CH}.
\end{proof}

Following \cite{BMNO}, for $\alpha$ close to zero, we denote $G(r, d, n; \alpha)$ by $\grdn$, and set
\[
S_0 (r, d, n) \ := \ \{ (E, V) \hbox{ $\alpha$-stable and generated of type } (r, d, n) \} \ \subseteq \ \grdn .
\]
By Lemma \ref{BgenOpen}, this is open in each component of $\grdn$.

\subsection{Projective bundles over a curve}

A vector bundle $E \to C$ gives rise to a projective bundle $\pe$, which we denote by $\rho \colon X \to C$. Over $X$ we have the relative hyperplane bundle $\Ox (1)$, which satisfies $\rho_* \Ox (1) \cong E$. Via adjunction, 
 we have a natural identification
\[
H^0 (C, E) \ \cong \ H^0 (X , \Ox (1) ) .
\]
Thus each subspace $V \subseteq H^0 (C, E)$ is canonically identified with a subspace of $H^0 (X, \Ox (1) )$, which we also denote by $V$. In this way, a coherent system $(E, V)$ over $C$ defines a projective bundle $X = \pe$ together with a linear system $(\Ox (1) , V)$ over $X$.

\begin{lemma} \label{c1Lr}
Let $C$, $E$ and $X$ be as above, and suppose that $V$ generates $\Ox (1)$. Then $c_1 ( \Ox (1) )^r = \deg E$.
\end{lemma}

\begin{proof}
Let $\psi \colon X \to \pv$ be the natural morphism. Choose general elements $H_1 , \ldots , H_r$ of $\PP V$, viewed as hyperplanes in $\pv$. Then
\[
c_1 ( \Ox (1) )^r \ = \ \# \left( H_1 \cap \cdots \cap H_r \cap \psi (X) \right) \ = \ \deg \psi (X) \ = \ \deg E . \qedhere
\]
\end{proof}

\section{Irreducibility of \texorpdfstring{$S_0 (r, d, n)$}{S\_0 (r, d, n)}} \label{section:Irred}

\noindent We continue to assume that $C$ is a smooth projective curve of genus $g \ge 3$.

\begin{lemma} \label{h1vanishing}
Suppose that $E \to C$ is a semistable bundle with $\mu (E) > 2g - 1$. 
\begin{enumerate}
\item[(a)] The bundle $E$ is generated.
\item[(b)] The cohomology group $H^1 (C, E)$ is zero.
\end{enumerate}
\end{lemma}

\begin{proof}
As $E$ is semistable, so are $E^*$ and $\Kc \otimes E^* (p)$ for any $p \in C$. Now
\[
\mu \left( \Kc \otimes E^* (p) \right) \ = \ 2g - 1 - \mu (E) \ < \ 0 ,
\]
and so $h^0 (C, \Kc \otimes E^* (p) ) = 0$. Thus by Serre duality $h^1 (C, E(-p)) = 0$, also, whence $H^0 (C, E) \to E|_p$ is surjective. Thus $E$ is generated. As $E|_p$ has support of dimension zero, $H^1 (C, E(-p)) \to H^1 (C, E)$ is surjective. Thus $h^1 (C, E) = 0$ also.
\end{proof}

\begin{proposition} \label{S0Irr}
Let $C$ be any curve of genus $g \ge 3$. Fix integers $r$, $d$ and $n$ satisfying
\[
r \ \ge \ 1 \quad \hbox{and} \quad d \ > \ r (2g - 1) \quad \hbox{and} \quad r + 1 \ \le \ n \ \le \ d - r(g - 1) .
\]
\begin{enumerate}
\item[(a)] The moduli space $\grdn$ is nonempty and irreducible, and contains $S_0 (r, d, n)$ as a dense open subset.
\item[(b)] The forgetful map $\Psi \colon S_0 (r, d, n) \dashrightarrow U_C (r, d)$ is a surjective morphism.
\end{enumerate}
\end{proposition}

\begin{proof}
(a) Let $(E_0, V_0 )$ be any $\alpha$-stable coherent system of type $(r, d, n)$. Then $E_0$ is semistable by Lemma \ref{alphasstability}. Now by \cite[Prop 2.6]{NR}, there exists an irreducible variety $T$ and a family $\cE \to T \times C$ containing $E_0$ and all stable bundles of rank $r$ and degree $d$. If necessary, we replace $T$ by the open subset $\{ t \in T : \cE_t \hbox{ semistable} \}$.

By Lemma \ref{h1vanishing} (b), the sheaf $\cP := \left( \pi_T \right)_* \cE$ is locally free of rank $d - r (g-1)$ over $T$. As $T$ is irreducible, so too is the Grassmann bundle $\Gr (n, \cP) \to T$.

Now there is a natural classifying map $\Phi \colon \Gr (n, \cP ) \dashrightarrow \grdn$. By Lemma \ref{alphasstability}, if $\cE_t$ is stable, then $\Phi$ is defined for all $V \in \Gr (n, \cP)|_t$. Therefore, by our choice of $T$ above, $\Phi \left( \Gr (n, \cP) \right)$ contains
\[
\{ (E, V) \in \grdn : E \hbox{ stable} \} \ := \ G' .
\]
The subset $\Phi^{-1} ( G' )$ is open because stability is an open condition, and therefore irreducible because $\Gr (n, \cP)$ is irreducible. Thus $G'$ is irreducible. By openness of stability, $G'$ is open in all components of $\grdn$. It follows that $G'$ is dense in a component of $\grdn$.

Now since $(E_0, V_0)$ was assumed to be $\alpha$-stable, $\Phi$ is also defined at $V_0 \in \Gr (n, \cP)|_{E_0}$. 
 Again using irreducibility of $\Gr (n, \cP)$, we see that $(E_0, V_0)$ belongs to the same component of $\grdn$ as the one containing $G'$. As $(E_0, V_0)$ was chosen to be an arbitrary $\alpha$-stable system, we conclude that $\grdn$ has no other irreducible component.

For the rest: By Lemma \ref{h1vanishing} (a) and since $d > r (2g - 1)$, for every $(E, V) \in \grdn$ the ambient bundle $E$ is generated. As $n \ge r + 1$, by Lemma \ref{BgenOpen} the subset $S_0 (r, d, n)$ of $\grdn$ is open and dense.

(b) By Lemma \ref{alphasstability}, if $(E, V)$ is $\alpha$-stable then $E$ is a semistable vector bundle. Thus the forgetful map $\Psi$ is defined everywhere. Furthermore, since $d \ge n + r(g - 1)$, we have $h^0 (C, E) \ge n$ for every $E$ of rank $r$ and degree $d$, and therefore $\Psi$ is surjective.
\end{proof}

\subsubsection*{Generality of quotient coherent systems}

The following will be required for the induction proof in the next section. 
 Fix integers $g \ge 3$ and $r \ge 2$ and $n \ge r + 1$ satisfying (\ref{NumHyp}). 

Using \cite[Proposition 2.6]{NR} as above, we may choose a family $\cF \to \bT \times C$ of vector bundles of rank $r - 1$ and degree $d$, where every stable bundle $F \in U_C (r - 1, d)$ is represented, and $\bT$ is irreducible. If necessary, we replace $\bT$ with the open subset parametrising stable $F$. As $d > 0$, by stability $h^0 (C, \cF^\vee_t ) = 0$ for all $t$, and so $R^1 \left( \pi_\bT \right)_* \cF^\vee$ is locally free over $\bT$. Let $R \to \bT$ be the associated projective bundle; this contains all nontrivial extensions
\begin{equation} \label{extE}
0 \ \to \ \Oc \ \to \ E \ \to \ \cF_t \ \to \ 0
\end{equation}
as $t$ ranges over $\bT$. Since $\bT$ is irreducible, so is $R$.

We claim now that a general element of $R$ is a stable vector bundle. Let $\cF_t$ be a general stable bundle, and let $N \in \Pic^0 (C)$ be general. Perturbing $\cF_t$ and $N$ if necessary, we may assume that $\cF_t \otimes N$ is also general in $U_C (r - 1, d)$. As $\deg N < \mu (F)$, by \cite{RT} a general extension $0 \to N \to E' \to \cF_t \otimes N \to 0$ is stable. But this is equivalent to stability of $0 \to \Oc \to E' \otimes N^{-1} \to \cF_t \to 0$. Thus we may assume that a general extension of the form (\ref{extE}) is a stable vector bundle over $C$.

Therefore, by Lemma \ref{h1vanishing}, a general $E \in R$ is generated and $h^1 (C, E) = 0$. We consider the open dense subset
\[
R' \ := \ \{ E \in R : h^1 (C, E) = 0 \} . 
\]
Now by \cite{Lan}, there is a universal extension $\cE \to R' \times C$ whose restriction to $\{ (\delta, t) \} \times C$ is isomorphic as a vector bundle to the extension $0 \to \Oc \to E \to \cF_t \to 0$ defined by the class $\delta \in \PP H^1 (C, \cF_t^\vee )$. By the vanishing of $h^1 (C, E)$, the sheaf $\cQ := \left( \pi_{R'} \right)_* \cE$ is locally free of rank $d - r(g - 1)$ over $R'$. We have a Grassmann bundle $\Gr (n, \cQ) \to R'$, with universal bundle $\cU \to \Gr (n, \cQ)$. As $R'$ is irreducible, so are $\Gr (n, \cQ )$ and the total space of $\cU$.

Now elements of $\cU \to \Gr (n, \cQ)$ correspond to triples $(E, V, \sigma)$ where $E$ is an extension of the form (\ref{extE}), and $V \subseteq H^0 (C, E)$ is a subspace of dimension $n$, and $\sigma \in V$. We define a rational map $\Gamma \colon \PP \cU \dashrightarrow S_0 (r - 1, d, n - 1)$ by
\begin{equation} \label{defnGamma}
(E, V, \sigma ) \ \mapsto \ \left( \frac{E}{\sigma ( \Oc )} , \frac{V}{\C \cdot \sigma} \right) .
\end{equation}
The point $\Gamma (E, V, \sigma )$ is defined if $\sigma$ is nowhere vanishing, $V$ generates $E$, and $E / \sigma ( \Oc )$ is a stable vector bundle.

\begin{proposition}
The map $\Gamma$ is dominant.
\end{proposition}

\begin{proof}
Let $(F, W)$ be a general element of $S_0 (r - 1, d, n - 1)$. It will suffice to show that there exists an extension $0 \to \Oc \to E \to F \to 0$ where $h^1 (C, E) = 0$ and $W \subseteq \Image \left( H^0 (C, E) \to H^0 (C, F) \right)$. (Note that any such $(E, V)$ is necessarily generated.)

Now $W \subseteq H^0 (C, F)$ lifts to an extension $E$ of the type considered if and only if the extension class $\delta (E)$ belongs to the kernel of the cup product map
\[
\cup_W \colon H^1 (C, F^\vee ) \ \to \ \Hom ( H^0 (C, F) , H^1 ( C, \Oc ) ) \ \to \ \Hom ( W, H^1 (C, \Oc ) ) .
\]
The kernel of $\cup_W$ has dimension at least $h^1 (C, F^\vee ) - (n - 1)g$, which by Riemann--Roch equals
\begin{equation} \label{dimKerCupW}
d + (r - 1)(g - 1) - (n - 1)g \ = \ d - (n - r)(g - 1) + n - 1 .
\end{equation}
In order to ensure that this contains a class $\delta (E)$ where $h^1 (C, E) = 0$, in view of Serre duality we will estimate the dimension of
\[
\{ \delta (E) \in H^1 (C, F^\vee ) : h^0 (C, \Kc \otimes E^\vee ) \ne 0 \} .
\]
As $h^1 (C, F) = 0$ by Lemma \ref{h1vanishing}, we have $h^0 (C, \Kc \otimes F^\vee ) = 0 $. Thus any nonzero section of $\Kc \otimes E^\vee$ corresponds to a lifting
\[ \xymatrix{
0 \ar[r] & F^\vee \ar[r] & E^\vee \ar[r] & \Oc \ar[r] & 0 \\
 & & & \Kc^\vee \ar[u]_\tau \ar@{..>}[ul] &
} \]
where $\tau \in H^0 (C, \Kc)$ is nonzero. This is equivalent to the condition that
\[
\delta (E) \ \in \ \Ker \left( \tau^* \colon H^1 (C, F^\vee ) \ \to \ H^1 (C, F^\vee \otimes \Kc ) \right) .
\]
The latter kernel is a quotient of $H^0 (C, F^\vee|_{(\tau)} )$, where $(\tau)$ is the divisor of $\tau$. Thus $\Ker ( \tau^* )$ has dimension at most $\rank F \cdot \deg \Kc = (r - 1)(2g - 2)$, and
\[
\bigcup_{(\tau) \in | \Kc |} \Ker (\tau^*)
\]
has dimension at most $(r - 1)(2g - 2) + (g - 1) = (2r - 1)(g - 1)$. A computation using (\ref{NumHyp}) shows that this is strictly less than (\ref{dimKerCupW}). 
 We conclude that there exists an extension $0 \to \Oc \to E \to F \to 0$ such that $W \subseteq \Image \left( H^0 (C, E) \to H^0 (C, F) \right)$ and $h^1 (C, E) = 0$. Let $\tW \subset H^0 (C, E)$ be a subspace mapping isomorphically to $W$. Write $\sigma$ for the injection $\Oc \isom \Ker (E \to F)$, and let $V$ be the span of $\tW$ and $\C \cdot \sigma$. Then $\Gamma (E, V, \sigma) = (F, W)$ (cf.\ (\ref{defnGamma})).

As $(F, W)$ was chosen to be generic, we conclude that $\Gamma$ is dominant.
\end{proof}

\begin{corollary} \label{GammaDom}
For a general $(E, V) \in S_0 (r, d, n)$ and general $\sigma \in V$, the quotient $\left( \frac{E}{\sigma ( \Oc )} , \frac{V}{\C \cdot \sigma} \right)$ is a general element of $S_0 (r - 1, d, n - 1)$.
\end{corollary}

\begin{proof}
Let $R'_\st \subseteq R'$ be the open dense sublocus of extensions $0 \to \Oc \to E \to \cF_t \to 0$ such that $E$ is stable. As $\Gamma$ is dominant, so is its restriction to the open dense sublocus
\begin{equation} \label{AuxLocus}
\PP \cU|_{\Gr \left( n, \cQ|_{R'_\st} \right)} \ \cong \ \{ (E, V, \sigma ) \in \PP \cU : E \hbox{ is stable} \}
\end{equation}
and the image by $\Gamma$ of a general element of this set can be assumed to be general in $S_0 (r - 1, d, n - 1)$.

But by Lemma \ref{alphasstability}, for any element $(E, V, \sigma)$ of (\ref{AuxLocus}), the system $(E, V)$ is $\alpha$-stable for $\alpha$ close to zero. For a general such triple, $V$ also generates $E$, and therefore $(E, V)$ defines an element of $S_0 (r, d, n)$. Thus the statement of the corollary holds already for those $(E, V) \in S_0 (r, d, n)$ where $(E, V)$ appears in $\Gr \left( n, \cQ|_{R'_\st} \right)$. As $S_0 (r, d, n)$ is irreducible by Proposition \ref{S0Irr} (a), the statement follows.
\end{proof}

\section{Stability of kernel bundles}

We are now in a position to prove our main result.

\begin{proof}[Proof of Theorem \ref{main}]
We proceed by induction on $r$. For $r = 1$: Let $(L, V)$ be a generated coherent system of type $(1, d, n)$ on $C$. Putting $E = L$, we have $\pe = C$ and $\opeo = L$. By \cite[Theorem 1.1]{BN}, if $(L, V)$ is general then $\mlv \to C$ is stable with respect to any polarisation on $C$ (cf.\ {\S} \ref{subsection:Stability}).

Suppose now that $r \ge 2$. By Proposition \ref{S0Irr} (a) \& (b), we may assume that if $(E, V)$ is general in the irreducible space $S_0 (r, d, n)$, then $E$ is general in moduli; and then clearly $V$ may be assumed to be a general subspace of $H^0 (C, E)$. Moreover, from \cite[Proposition 2.3.1]{HL} it follows that the locus
\begin{equation} \label{StableInS0}
\left\{ (E, V) \in S_0 (r, d, n) : \moev \hbox{ is $\opeo$-stable} \right\}
\end{equation}
is open in $S_0 (r, d, n)$. Thus it will suffice to prove that (\ref{StableInS0}) is nonempty.

Let $(E, V)$ be a general element. We will adapt a strategy employed in \cite[{\S} 4]{Cam} and \cite[pp.\ 74-75]{ELM}, using also results of \cite{Mis}, to show that $\moev$ is $\opeo$-stable. To ease notation, we write $X := \pe$ and $L := \opeo$, so that $\moev = \mlv$. Let $G \subset \mlv$ be a torsion free subsheaf.

We now consider hyperplane sections $Y$ of $X$ defined by general elements $\sigma \in V$. Firstly, as $V$ generates $E$, as a subspace of $H^0 (X, L)$ also $V$ generates $L$. Thus by Bertini's theorem, we may assume that
\begin{equation} \label{one}
\hbox{a general hyperplane section $Y \in \pv$ is smooth and integral.}
\end{equation}
Therefore, by the argument given in \cite{MooS}, the restricted map $G|_Y \to \mlv|_Y$ is again injective. As $\mlv$ is locally free, $\mlv|_Y$ is locally free on $Y$, and in particular torsion free. By the aforementioned injectivity of $G|_Y \to \mlv|_Y$,
\begin{equation} \label{two}
\hbox{the restriction $G|_Y$ to a general hyperplane section $Y$ is torsion free.}
\end{equation}
In particular, $G|_Y$ is locally free on $Y$ away from a subset of codimension at least $2$. Write $\iota \colon Y \hookrightarrow X$ for the inclusion. Unwinding the definitions of $\det G$ and $\det (G|_Y)$ (using the fact that $Y$ and $X$ are integral and locally factorial; see for example \cite[p.\ 129]{Har}), we see that $\det (G|_Y) \cong \iota^* (\det G)$, and in particular $c_1 (G|_Y) = \iota^* c_1 (G)$. Following the argument of \cite[Lemma 2.1]{Mis}, we now obtain for a general $Y \in \PP V$ that
\begin{equation} \label{SlopesEqual}
\deg_{\cL} (G) \ = \ \int_X c_1 (G) \cdot c_1 ( \cL )^{r-1} \ = \ \int_Y c_1 \left( G|_Y \right) \cdot c_1 ( L|_Y )^{r-2} \ = \ \deg_{L|_Y} \left( G|_Y \right) .
\end{equation}
(Note that when $r = 2$, so that $Y$ is a curve, the above reduces to
\[
\deg_{\cL} (G) \ = \ \int_X c_1 (G) \cdot c_1 ( \cL ) \ = \ \int_Y c_1 \left( G|_Y \right) \ = \ \deg \left( G|_Y \right) .)
\]

Next, by (\ref{NumHyp}), in particular $d \ge (n + r - 3)(g - 1) + n - 2$. 
 Thus, by induction and by Corollary \ref{GammaDom}, we may assume that
\begin{multline} \label{three}
\hbox{for a general $(E, V) \in S_0 (r, d, n)$ and general $\sigma \in V$, the kernel bundle} \\
\hbox{$\mofw$ is $\opfo$-stable, where } (F, W) = \left( \frac{E}{\sigma (\Oc )} , \frac{V}{\C \cdot \sigma} \right) .
\end{multline}

Let now $x \in X$ be a general point. Then we may assume that $G$ is locally free at $x$. Let $Y \in \PP V$ be the divisor defined by an element
\[
\sigma \ \in \ \mlv|_x \ = \ V \cap H^0 ( X , L \otimes \cI_x )
\]
satisfying the generality conditions (\ref{one}), (\ref{two}) and (\ref{three}), and also lying outside the proper subspace $G|_x \subset \mlv|_x$. By (\ref{one}), as a variety, $Y \cong \pf$ where $F$ is the vector bundle defined by $0 \to \Oc \xrightarrow{\sigma} E \to F \to 0$. Set $W := V / \C \cdot \sigma$. 

Since $\sigma$ vanishes on $Y$, noting that $\opeo|_Y = \opfo$, we obtain an exact diagram of sheaves on $Y$:
\[ \xymatrix{
\Oy \otimes \C \cdot \sigma \ar[d] \ar[r]^\sim & \Oy \otimes \C \cdot \sigma \ar[d] & \\
\mlv|_Y \ar[r] \ar[d] & ( \Ox \otimes V)|_Y \ar[r] \ar[d] & L|_Y \ar[d]^\wr \\
\mofw \ar[r] & \Oy \otimes W \ar[r] & \opfo
} \]
As $\sigma$ lies outside $G|_x$, the composed map $G|_Y \to \mlv|_Y \to \mofw$ is injective. Now $c_1 ( \mofw ) = - c_1 ( \opfo )$. By the stability hypothesis (\ref{three}), applying Lemma \ref{c1Lr} to both $(E, V)$ and $(F, W)$ we obtain
\begin{multline*}
\mu_{\opfo} \left( G|_Y \right) \ \le \ \mu_{\opfo} ( \mofw ) \ = \ -\frac{c_1 ( \opfo )^{r-1}}{\dim W - 1} \ = \\
 -\frac{d}{n-2} \ < \ -\frac{d}{n-1} \ = \ -\frac{c_1 ( L )^r}{n-1} \ = \ \mu_L ( \mlv ) .
\end{multline*}

In view of (\ref{SlopesEqual}), this implies that $\mu_L (G) < \mu_L (\mlv)$. It follows that $\mlv$ is $\cL$-stable, as desired.
\end{proof}


\begin{thebibliography}{99}


\bibitem [Bea03] {Bea} A.\ Beauville: \textsl{Some stable vector bundles with reducible theta divisor}, Manuscr.\ Math.\ \textbf{110}, no.\ 3 (2003), 343--349.

\bibitem [BBN15] {BBN} U.\ N.\ Bhosle, L.\ Brambila-Paz and P.\ E.\ Newstead: \textsl{On linear series and a conjecture of D.\ C.\ Butler}, Int.\ J.\ Math.\ \textbf{26}, no.\ 2 (2015), Article ID 1550007, 18 pp.

\bibitem [BMNO19] {BMNO} L.\ Brambila-Paz, O.\ Mata-Guti\'errez, P.\ E.\ Newstead and A.\ Ortega: \textsl{Generated coherent systems and a conjecture of D.\ C.\ Butler}, Int.\ J.\ Math.\ \textbf{30}, no.\ 5 (2019), Article ID 1950024, 25 pp.

\bibitem [BN26] {BN} L.\ Brambila-Paz and P.\ E.\ Newstead: \textsl{New examples of twisted Brill--Noether loci II}, arXiv:2601.11855 (2026).

\bibitem [But94] {But94} D.\ C.\ Butler: \textsl{Normal generation of vector bundles over a curve}, J.\ Diff.\ Geom.\ \textbf{39}, no.\ 1 (1994), 1--34.

\bibitem [But97] {But97} D.\ C.\ Butler: \textsl{Birational maps of moduli of Brill--Noether pairs}, arXiv:alg-geom/9705009 (1997).

\bibitem [CL21] {CL} F.\ Caucci and M.\ Lahoz: \textsl{Stability of syzygy bundles on abelian varieties}, Bull.\ Lond.\ Math.\ Soc.\ \textbf{53}, no.\ 4 (2021), 1030--1036.

\bibitem [Cam12] {Cam} C.\ Camere: \textsl{About the stability of the tangent bundle of $\PP^n$ restricted to a surface}, Math.\ Z.\ \textbf{271}, no.\ 1--2 (2012), 499--507.

\bibitem [CH24] {CH} A.\ Castorena and G.\ H.\ Hitching: \textsl{On $\alpha$-stability and linear stability of generated coherent systems}, arXiv:2409.12794 (2024); to appear in Kyoto J.\ of Math.

\bibitem [CH10] {ChH} I.\ Choe and G.\ H.\ Hitching: \textsl{Secant varieties and Hirschowitz bound on vector bundles over a curve}, Manuscripta Math.\ \textbf{133}, no.\ 3--4 (2010), 465--477.

\bibitem [ELM13] {ELM} L.\ Ein, R.\ Lazarsfeld and Y.\ Mustopa: \textsl{Stability of syzygy bundles on an algebraic surface}, Math.\ Res.\ Lett.\ \textbf{20}, no.\ 1 (2013), 73--80.

\bibitem [Ful98] {Ful} W.\ Fulton: \textsl{Intersection theory}, 2nd.\ ed., Ergebnisse der Mathematik und ihrer Grenzgebiete, 3.\ Folge 2. Berlin: Springer, 1998.



\bibitem [Har89] {Har} R.\ Hartshorne: \textsl{Stable reflexive sheaves}, Math.\ Ann.\ \textbf{254} (1980), 121--176.

\bibitem [HL10] {HL} D.\ Huybrechts and M.\ Lehn: \textsl{The geometry of moduli spaces of sheaves}, 2nd ed. Cambridge: Cambridge University Press, 2010.

\bibitem [JR25] {JR} C.\ Jiang and P.\ Ren: \textsl{Stability of syzygy bundles on varieties of Picard number one}, Int.\ Math.\ Res.\ Not.\ 2025, no.\ 8, article ID rnaf098, 11 pp.

\bibitem [KN95] {KN} A.\ D.\ King and P.\ E.\ Newstead: \textsl{Moduli of Brill--Noether pairs on algebraic curves}, Int.\ J.\ Math.\ \textbf{6}, no.\ 5 (1995), 733--748.

\bibitem [Kle74] {Kle} S.\ L.\ Kleiman: \textsl{The transversality of a general translate}, Compos.\ Math.\ \textbf{28} (1974), 287--297.

\bibitem [Lan83] {Lan} H.\ Lange: \textsl{Universal families of extensions}, J.\ Algebra \textbf{83} (1983), 101--112.



\bibitem [MS23] {MS} R.\ M.\ Mir\'o-Roig and M.\ Salat-Molt\'o: \textsl{Syzygy bundles of non-complete linear systems: stability and rigidness}, Mediterr.\ J.\ Math.\ \textbf{20} (2023), no.\ 5, paper no.\ 265, 21 pp.

\bibitem [Mis21] {Mis} E.\ C.\ Mistretta: \textsl{A remark on stability and restrictions of vector bundles to hypersurfaces}, Bol.\ Soc.\ Mat.\ Mex., III.\ Ser.\ \textbf{27}, no.\ 2 (2021), Paper No. 36, 7 pp.

\bibitem [MooS17] {MooS} \texttt{https://math.stackexchange.com/questions/2440449/restriction-of-an-injection-of-sheaves}, 2017.

\bibitem [NR75] {NR} M.\ S.\ Narasimhan and S.\ Ramanan: \textsl{Deformations of the moduli space of vector bundles over an algebraic curve}, Ann.\ Math.\ (2) \textbf{101} (1975), 391--417.

\bibitem [Popa99] {Popa} M.\ Popa: \textsl{On the base locus of the generalized theta divisor}, C.\ R.\ Acad.\ Sci., Paris, S\'er.\ I, Math.\ \textbf{329}, no.\ 6 (1999), 507--512.

\bibitem [Rek24] {Rek} N.\ Rekuski: \textsl{Stability of kernel sheaves associated to rank one torsion-free sheaves}, Math.\ Z.\ \textbf{307}, no.\ 1 (2024), paper no.\ 2, 18 pp.

\bibitem [RT99] {RT} B.\ Russo and M.\ Teixidor i Bigas: \textsl{On a conjecture of Lange}, J.\ Algebr.\ Geom.\ \textbf{8}, no.\ 3 (1999), 483--496.

\bibitem [Sar26] {Sar} S.\ Sarkar: \textsl{Stable and unstable syzygy bundles on certain Picard rank one varieties}, arXiv:2607.10585 (2026).

\bibitem [TZ22] {TZ} H.\ Torres-L\'opez, and A.\ G.\ Zamora: \textsl{$H$-stability of syzygy bundles on some regular algebraic surfaces}, Beitr.\ Alg.\ Geom.\ \textbf{63}, no.\ 3 (2022), 589--598.

\end{thebibliography}
\end{document}